\documentclass[11pt,twoside]{amsart}
\usepackage{latexsym,amssymb,amsmath}

\numberwithin{equation}{section}
\newtheorem{Theorem}{Theorem}[section]
\newtheorem{Lemma}[Theorem]{Lemma}
\newtheorem{Corollary}[Theorem]{Corollary}
\newtheorem{Proposition}[Theorem]{Proposition}

\theoremstyle{definition}
\newtheorem{Definition}[Theorem]{Definition}

\newtheorem{Example}[Theorem]{Example}

\def\deg{\operatorname{degree}}

\def\height{\operatorname{height}}
\def\depth{\operatorname{depth}}
\def\sdepth{\operatorname{s-depth}}
\def\C{{\mathfrak C}}
\def\e{{\mathfrak e}}

\begin{document}

\title[Depths and Cohen-Macaulay Properties of Path Ideals]{Depths and Cohen-Macaulay Properties of Path Ideals}

\author{Daniel Campos}
\address{Department of Mathematics \\
Texas State University\\
601 University Drive\\ 
San Marcos, TX 78666.}
\email{campos.daniell@gmail.com}
\thanks{All authors were partially supported by NSF (DMS 1005206)}

\author{Ryan Gunderson}
\address{Department of Mathematics \\
University of Nebraska\\
Lincoln, NE 68588}
\email{ryan.gunderson@yahoo.com}

\author{Susan Morey}
\address{Department of Mathematics \\
Texas State University\\
601 University Drive\\ 
San Marcos, TX 78666.}
\email{morey@txstate.edu}

\author{Chelsey Paulsen}
\address{Department of Mathematics \\
North Dakota State University
Fargo, ND 58108}
\email{chelsey.paulsen@gmail.com}

\author{Thomas Polstra}
\address{Department of Mathematics\\
202 Mathematical Sciences Bldg.\\
University of Missouri\\
Columbia, MO 65211}
\email{thomaspolstra@gmail.com}


\begin{abstract} Given a tree $T$ on $n$ vertices, there is an associated ideal $I$ of $R[x_1, \ldots, x_n]$ generated by all paths of a fixed length $\ell$ of $T$. We show that such an ideal always satisfies the K\"{o}nig property and classify all trees for which $R/I$ is Cohen-Macaulay. More generally, we show that an ideal $I$ whose generators correspond to any collection of subtrees of $T$ satisfies the K\"{o}nig property. Since the edge ideal of a simplicial tree has this form, this generalizes a result of Faridi. Moreover, every square-free monomial ideal can be represented (non-uniquely) as a subtree ideal of a graph, so this construction provides a new combinatorial tool for studying square-free monomial ideals. For a special class of trees, namely trees that are themselves a path of length at least $l$, a precise formula for the depth is given and it is shown that the proof extends to provide a lower bound on the Stanley depth of these ideals. Combining these results gives a new class of ideals for which the Stanley Conjecture holds.  
\end{abstract}

\maketitle

\section{Introduction}

There is a well-known correspondence between square-free monomial ideals generated in degree two and graphs. If $G$ is a graph on $n$ vertices, let $R=k[x_1, \ldots , x_n]$ be a polynomial ring over a field $k$ in $n$ variables and define the {\em edge ideal} $I=I(G)$ to be the ideal generated by all monomials of the form $x_ix_j$ where $\{x_i, x_j\}$ is an edge of $G$. Introduced by Villarreal in 1990 \cite{VillaCMGraphs}, the use of graphs to study algebraic properties of edge ideals has proven quite fruitful. Various authors have used combinatorial information from the associated graph to deduce information about associated primes \cite{AJ, MMV, SVV}, depths \cite{HH, morey}, regularity \cite{kummini}, and other invariants. This correspondence has been extended to arbitrary square-free monomial ideals by using clutters \cite{clutters}, hypergraphs \cite{Ha-VanTuyl}, or facets of simplicial complexes  \cite{FaridiSimplicial}.

In this paper, we introduce a new way to represent the combinatorial properties of a square-free monomial ideal. Given any square-free monomial ideal $I$, there is a graph $G$ such that the generators of $I$ correspond to the vertices of subtrees of $G$. The graph $G$ is far from unique, yet this approach yields a surprising amount of information about the ideal. The case where $G$ can be chosen to be a tree is of particular interest. The class of ideals whose generators are subtrees of a tree is a broad class encompassing several previously studied classes of graphs. For example, if $T$ is a tree and $\ell$ is a positive integer, the ideal whose generators correspond to paths of $T$ of length $\ell$ is called a {\em path ideal} of $T$ and is denoted by $I_{\ell}(T)$.  If $T$ is a directed or rooted tree, one can similarly form a path ideal where the generators of $I$ correspond to directed paths of $T$. Path ideals were introduced in \cite{ConcaDeNegri} and have been studied by various authors. Since a path is a special type of subtree, path ideals, of directed and of undirected trees, fall into this new structure. Edge ideals of simplicial trees fall into this class as well, as will be seen in Theorem~\ref{Simplicial Trees are Subtree Clutters}.

The primary goal of this paper will be to examine depth properties of path ideals of trees. Depth properties of edge ideals of trees are reasonably well understood (see \cite{morey}) and so this is a natural extension. While many authors have focused on path ideals of directed trees, the trees in this paper will be undirected. Such ideals can be more complicated due to the larger number of generators, however, using undirected paths yields a particularly nice classification of when the depth is maximal, that is, when the ideal is Cohen-Macaulay. 

As a consequence of the method of proof employed, we are able to extend the information found regarding the depth of the path ideal of a special type of tree, namely a tree that is itself a path, to a lower bound on the Stanley depth of such trees. For an introduction to Stanley depths, see \cite{WhatIsStanleyDepth}. Let $I$ be a monomial ideal. A {\em Stanley decomposition} of $R/I$ is a direct sum decomposition $R/I = \oplus_{i=1}^s m_iR_{t_i}$ where $m_i$ is a monomial and $R_{t_i}=k[x_{i_1}, \ldots ,x_{i_{t_i}}]$ is a polynomial subring of $R$ generated over $k$ by $t_i$ of the variables of $R$. The depth of this decomposition is the minimum of the $t_i$, that is, the smallest number of variables used in any summand. The {\em Stanley depth}, denoted $s$-depth, of $R/I$ is then the maximum depth of a Stanley decomposition of $R/I$.  Introduced in \cite{StanleyDepth}, this is a more geometric invariant attached to a monomial ideal, or more generally to a ${\mathbb Z}^r$ graded module. Stanley conjectured that the Stanley depth is always bounded below by the depth. By combining the bound found in Theorem~\ref{StanleyDepths} with Theorem~\ref{Depth of Spines}, we prove that one class of path ideals is Stanley, that is, the Stanley Conjecture holds true for this class of ideals. While other classes of Stanley ideals are known, see for instance \cite{StanleyDepthTree} or \cite{Cimpoeas}, the conjecture is still largely open.

The contents of the paper are as follows. In Section~\ref{Defns} we provide the definitions and basic facts used throughout paper. In Section~\ref{CMPathIdeals}, we introduce subtree ideals. In Theorem~\ref{Trees are Konig} we show that every subtree ideal of a tree satisfies the K\"{o}nig property. Theorem~\ref{Simplicial Trees are Subtree Clutters} shows that the ideal of a simplicial tree is a subtree ideal, and Example~\ref{Tree not Simplicial} shows that the converse does not hold. Thus Theorem~\ref{Trees are Konig} extends \cite[Theorem 5.3]{FaridiCM} to a larger class of ideals.  Using this, we classify all Cohen-Macaulay path ideals of trees. That is, we show in Theorem~\ref{CM iff Suspension} that a path ideal of length $\ell$ of a tree $T$ is Cohen-Macaulay if and only if $T$ is a suspension of length $\ell$ of another tree. This generalizes \cite[Theorem 2.4]{VillaCMGraphs} from edge ideals to path ideals. 

In Section~\ref{Depths and Stanley} we focus on the particular case where the tree $T$ is a path. An exact formula for the depths of the path ideals is computed in Theorem~\ref{Depth of Spines}. This formula, together with the Auslander-Buchsbaum formula, recovers the projective dimesions found in \cite[Theorem 4.1]{HeVanTuyl} and recovered in \cite[Corollary 5.1]{BouchatHaO'Keefe}. However, the method of proof involves applying the Depth Lemma repeatedly to a series of short exact sequences. This method of proof was shown to extend to provide a bound on the Stanley depth of the ideal in \cite{StanleyDepthTree}. Such a bound is given in Theorem~\ref{StanleyDepths} and as a result, in Corollary~\ref{StanleyIdeals} these ideals are seen to be Stanley, that is, the Stanley Conjecture is satisfied for this class of ideals.

\section{Definitions and Background}\label{Defns}

We begin by reviewing some standard notation and terminology regarding graphs and simplicial complexes and their connections to algebra. Note that by abuse of notation, $x_i$ will be used to denote both the vertex of a graph $G$ and the corresponding variable of the polynomial ring $R$. For additional information regarding monomial ideals, see \cite{monalg} and for additional background in graph theory, see \cite{Har}.

A {\em graph} is a vertex set $V=\{x_1, \ldots ,x_n\}$ together with a set $E=E(G) \subseteq V \times V$ of edges. As previously stated, associated to any graph $G$ is a square-free monomial ideal generated in degree two, $I=I(G)$ called the edge ideal of $I$. To generalize this correspondence to monomial ideals with generators of degree greater than two, researchers have used simple hypergraphs, clutters, or simplicial complexes. A \emph{clutter} $\mathfrak{C}$ is a vertex set $V$ together with a family of subsets of $V$, called edges, none of which are included in one another.  As with edge ideals of graphs, there is a one-to-one correspondence between clutters and square-free monomial ideals. If $\mathfrak{C}$ is a clutter, then $I(\mathfrak{C})$ is the ideal whose generators are the products of the vertices in each edge of $\mathfrak{C}$. Throughout the paper, we will sometimes abuse notation by using the edge $e$ of a clutter $\C$ interchangeably with the generator $x_e=\prod_{ x_j \in e} x_j$ of $I(\C)$. We will also write $e_1 \cap e_2$ to mean the intersection of the supports of the edges, that is, the set of vertices that appear in both edges.

Some basic notions from graph theory will be used throughout the paper and so are presented here for completeness.  
If $V' \subset V$ is a subset of the vertices of a graph $G$, the {\em induced subgraph} on $V'$ is the graph $G'$ given by $V(G')=V'$ and $E(G')=\{ e \in E \, | \, e \subset V'\}$. That is, the edges of $G'$ are precisely the edges of $G$ with both endpoints in $V'$. The induced subclutter of a clutter is defined similarly. If $x \in V(G)$, the {\em neighbor set} $N(x)$ is the set of all vertices that are adjacent to $x$, that is, $N(x) = \{ y \in V(G) \, | \, \{x,y\} \in E(G)\}$. The {\em degree} of a vertex $x$ is the cardinality of $N(x)$. A {\em leaf} is a vertex of degree one, and a {\em tree} is a connected graph where every induced subgraph has a leaf. A walk of length $s$ is a collection of vertices and edges $x_0, e_1, x_1, e_2, \ldots , e_s, x_s$ where $e_i=x_{i-1}x_i$ for $1 \leq i \leq s$. A walk without repeated vertices is a {\em path}. A walk where $x_0=x_s$ but no other vertices are repeated is a {\em cycle}. It is easy to see that a tree is a graph with no cycles and if $T$ is a tree, then for any vertices $x,y \in V(G)$ there is a unique path between $x$ and $y$. The length of this path is the {\em distance} between $x$ and $y$, which is denoted by $d(x,y)$. In a general graph, $d(x,y)$ is the minimum of the lengths of all paths connecting $x$ and $y$. A {\em forest} is a collection of trees. An {\em isolated vertex} is a vertex $x$ with $N(x)=\emptyset$. Since $k[x_1, \ldots , x_n, y]/(I,y) \cong k[x_1, \ldots , x_n]/I$ for any monomial ideal $I$ whose generators lie in $k[x_1, \ldots, x_n]$, the graphs throughout this paper are generally assumed to be free of isolated vertices.

There are two common constructions used to produce smaller, related graphs from a fixed graph that will be useful throughout the paper. Both extend naturally to clutters. One is the {\em deletion} $G\setminus x$, which is formed by removing $x$ from the vertex set of $G$ and deleting any edge in $G$ that contains $x$. This has the effect of setting $x=0$, or of passing to the quotient ring $R/(x)$. The other operation is the {\em contraction}, $G/x$. This is performed by removing $x$ from the vertex set and removing $x$ from any edge that contains $x$. When $G$ is a graph, this will result in each vertex in $N(x)$ becoming an isolated vertex. This operation has the effect of setting $x=1$, or of passing to the localization $R_x$. A {\em minor} of a graph or clutter is formed by performing any combination of deletions and contractions.

If $G$ is a graph or $\C$ is a clutter, a {\em minimal vertex cover} of $G$ or $\C$ is a set $C \subset V$ such that for every $e \in E$, $e \cap C \neq \emptyset$ and $C$ is minimal with respect to this property, meaning if $C'$ is any proper subset of $C$, then there exists an edge $e \in E$ with $e \cap C' = \emptyset$. The minimum cardinality of a minimal vertex cover of $G$ (or $\C$) is denoted by $\alpha_0=\alpha_0(G)$. A prime ideal $P$ is a {\em minimal prime} of an ideal $I$ if $I \subset P$ and if $Q$ is a prime ideal with $I \subset Q \subset P$, then $Q=P$. It is straightforward to check that $C$ is a minimal prime of $G$ or $\C$ if and only if the prime ideal $P$ generated by the variables corresponding to vertices of $C$ is a minimal prime of $I(G)$ or $I(\C)$. Thus $\alpha_0=\height(I)$. Two basic facts about minimal vertex covers that will be used throughout the paper are that if $x \in V$ then there exists a minimal vertex cover $C$ with $x \in C$, and if $C'$ is a minimal vertex cover for an induced subgraph (or subclutter) $H$ of $G$, then there exists a minimal vertex cover $C$ with $C' \subset C$. To see why this second fact holds, consider $C' \cup (V(G)\setminus V(H))$. This is a vertex cover, and so it contains a minimal vertex cover. However, deleting any vertex of $C'$ would leave an edge of $H$ uncovered, so the minimal vertex cover produced must contain $C'$. 

Given a graph $G$, there is another family of square-free monomial ideals associated to $G$. For each positive integer $\ell$, define $P_{\ell}(G)$ to be the monomial ideals whose generators correspond to paths of length $\ell$ of $G$. Notice that a path of length $\ell$ contains $\ell +1$ vertices, so $P_{\ell}(G)$ is a homogeneous ideal with generators of degree $\ell +1$. When $\ell = 1$, $P_{1}(G)=I(G)$ is the edge ideal of $G$. Notice that since a path is defined to have distinct vertices, $P_{\ell}(G)$ is a square-free monomial ideal. We will primarily be interested in the case where $G$ is a tree. 

A set of edges $X \subset E$ of a graph or clutter is an {\em independent} set if for every $e_1,e_2 \in X$, $e_1 \cap e_2 = \emptyset$. A maximal set of independent edges is called a {\em matching}, and the maximum cardinality of a maximal matching is denoted by $\beta_1$.  As in \cite{GRV}, this corresponds to the {\em monomial grade} of the associated ideal $I$. Since every minimal vertex cover must contain at least one vertex from each edge in an independent set, $\alpha_0 \geq \beta_1$. The case of equality is of particular interest.
\begin{Definition}\label{Konig}
A clutter $\mathfrak{C}$ is said to satisfy the K\"{o}nig property if the cardinality of a maximum set of independent edges, $\beta_{1}$, is equal to $\alpha_{0}$, the cardinality of a minimal vertex cover.  
\end{Definition}
Combinatorially, the K\"{o}nig property is a generalization of the bipartite condition on a graph. A graph is {\em bipartite} if it has no odd cycles. Bipartite graphs are graphs for which every minor satisfies the K\"{o}nig property. Note that this property is also of interest algebraically. We have that $\C$ satisfies the K\"{o}nig property if and only if the height of $I(\C)$ is equal to the monomial grade of $I(\C)$. 

\begin{Definition} \label{Perfect Matching of Konig Type Definition}
A clutter $\C$ has a \emph{perfect matching of K\"{o}nig type} if there is a collection $e_1, \dots, e_g$ of pairwise disjoint edges, with $g=\height(I(\mathfrak{C}))$,  whose union is $V(\C)$. 
\end{Definition}

Edge ideals of clutters are sometimes viewed as facet ideals of simplicial complexes. Recall that a simplicial complex $\Delta$ is a set of vertices $V$ together with a collection of subsets of $V$, called faces, such that every subset of a face is a face. A facet is a face that is maximal with respect to inclusion. The ideal of a simplicial complex $I(\Delta)$ is the square-free monomial ideal generated by the monomials corresponding to the facets of $\Delta$. In \cite{FaridiSimplicial}, Faridi introduced the concept of a {\em simplicial tree}. A facet $F$ of $\Delta$ is a {\em leaf} if there exists a facet $G$ of $\Delta$ such that $F \cap H \subset F \cap G$ for every facet $H$ of $\Delta \setminus \{F\}$. A simplicial tree is then a simplicial complex for which every subcomplex contains a leaf. In \cite{HHTZ} a {\em special cycle} of a simplicial complex was defined to be an alternating collection of vertices and facets $x_0, F_1, x_1, F_2, \ldots , x_{s-1}, F_s, x_s$ such that $x_0=x_s$, $x_{i-1 }, x_i\in F_i$ and $x_j \not\in F_i$ for $j \neq i-1, i$. It was shown in \cite[Theorem 3.2]{HHTZ} that a simplicial tree does not contain any special odd cycles of length greater than $2$. It was also shown in \cite[Theorem 2.7]{HeVanTuyl} that a path ideal of a directed tree corresponds to a simplicial tree. This correspondence does not extend to path ideals of more general trees, as will be seen in Example~\ref{Tree not Simplicial}.

If $\Delta$ is a simplicial complex, the $1-$skeleton of $\Delta$, denoted by $\Delta^1$ is defined to be the graph on the vertices of $\Delta$ whose edges are the faces of $\Delta$ that contain precisely two elements. Note that by the definition of a simplicial complex, the $1-$skeleton of a facet of $\Delta$ will be a complete graph on the vertices of the facet. For a connected graph $G$, a {\em spanning tree} is defined to be a subgraph $T$ of $G$ such that $T$ is a connected tree and $V(T)=V(G)$. When $\Delta$ is a simplicial complex, a spanning tree of $\Delta$ is defined to be a spanning tree of $\Delta^1$. 

There is another simplicial complex, called the Stanley-Reisner complex, associated to a square-free monomial ideal that is in a sense dual to the one discussed above. This complex is denoted by $\Delta_\mathfrak{C}$. It is the complex whose faces are the independent vertex sets of $\C$.

\section{Cohen-Macaulay Path Ideals}\label{CMPathIdeals}

If $I$ is the edge ideal of a tree $T$, then by \cite[Theorem 2.4]{VillaCMGraphs}, $I$ is Cohen-Macaulay if and only if $T$ is the suspension of a subtree $T'$. The primary purpose of this section is to extend this result to path ideals of trees. In order to show the desired result, we will need to prove that path ideals of trees satisfy the K\"{o}nig property. This result holds in greater generality. Instead of requiring the ideal to be generated by all paths of a fixed lenght $\ell$, the K\"{o}nig property will hold for ideals generated by a subset of the set of paths of a fixed length, by paths of different lengths, or by monomials corresponding to subtrees of the tree that are not necessarily paths. In order to prove this more general result, we first need to introduce terminology and notation allowing for the more general generating sets.

Recall that a subtree $T'$ of a tree $T$ is a connected induced subgraph of $T$. If $F$ is a forest, a subtree of $F$ is a subtree of one of the connected components of $F$. 
Let $F$ be a forest with vertex set $V$. A \emph{subtree clutter} of $F$ is a clutter $\C$ such that $V(\C)=V(F)$ and if $\e\in E(\C)$ is an edge of $\C$ then $\e$ is a subtree of $F$. We define a {\em subtree ideal} $I=I(\C)$ to be the ideal generated by the square-free monomials corresponding to the generators of a subtree clutter $\C$.  This more general class of ideals is quite interesting and encompasses several known classes of ideals, including path ideals of trees and ideals of simplicial trees, as will be seen in Theorem~\ref{Simplicial Trees are Subtree Clutters}. Notice that if $F$ is a forest with vertex set $V$, $\mathfrak{C}$ a subtree clutter of $F$, and $v\in V$, then $\mathfrak{C}\setminus\{v\}$ is a subtree clutter for the forest $F\setminus\{v\}$. More generally, if $F'$ is an induced subgraph of $F$ with vertices in a set $V'$, define $\C \cap F' = \{ e \in \C \, | \, {\mbox{\rm the vertices of }} e \, {\mbox{\rm  are in }} V'\}$. Then $\C \cap F'$ is a subtree clutter of $F'$. 
Thus if $F=T_1\cup \cdots \cup T_n$ where $T_i$ are the trees of $F$ and $E$ is any set of subtrees of $F$, then $E_{T_i}=E \cap T_i$ is a set of subtrees of $T_i$. Moreover, $E= E_{T_1} \cup \cdots \cup E_{T_n}$ and if $E$ is an independent set, then so is $E_{T_i}$ for each $i$.

\begin{Definition}\label{Maximal}
Let $\mathfrak{C}$ be a subtree clutter of a forest $F$, let $E$ be a set of independent edges of $\mathfrak{C}$, and let $T$ be a tree of $F$. We say that $E_{T}$ is \emph{maximal} if there does not exist a set of independent edges in $\mathfrak{C} \cap T$  with cardinality larger than $|E_{T}|$.
\end{Definition}

Before proving the main result, we give two lemmas that provide basic information about independent sets. The first will allow us to focus on a particular tree within a forest, while the second shows how independent sets relate to deletion minors of a tree. Note that a minor of a tree is a forest.

\begin{Lemma}\label{Maximal Edges}
Let $\mathfrak{C}$ be a subtree clutter of a forest $F$ whose connected components are $T_{1},T_{2},...,T_{n}$. Then $E$ is a set of independent edges of maximum cardinality of $\mathfrak{C}$ if and only if $E_{T_{i}}$ is maximal for all $1\leq i \leq n$.
\end{Lemma}

\begin{proof}
Suppose $E_{T_{i}}$ is not maximal for some $i$. Then there exists an $E_{T_{i}}'$ of larger cardinality than $|E_{T_{i}}|$ whose edges contain vertices only from the tree $T_{i}$. Thus the set $E_{T_{1}} \cup\cdots \cup E_{T_{i-1}}\cup E_{T_{i}}'\cup E_{T_{i+1}}\cup\cdots \cup E_{T_{n}}$ is a set of independent edges of $\mathfrak{C}$ cardinality greater than $|E_{T_{1}}\cup \cdots \cup E_{T_{i}}\cup  \cdots \cup E_{T_n}|=|E|$, a contradiction. 
 
Since $E=\bigcup_{i=1}^{n}E_{T_{i}}$, the converse is easily verified. 
\end{proof}

\begin{Lemma}\label{Vertex Missing 2}
Let $\C$ be a subtree clutter of a tree $T$ with $V=V(T)$. Let $E_T$ be a set of independent edges of $\C$, and let $\mathfrak{e}\in E_T$ be a subtree of $T$. Let $v$ be a vertex of $\mathfrak{e}$. Then $E_{T}=E_{T_{1}}\cup E_{T_{2}}\cup\cdots \cup E_{T_{n}}\cup \{ \mathfrak{e} \}$ where $T_1, \ldots , T_n$ are the connected components of $T\backslash \{v\}$ and $E_{T_i}=E_T\cap T_i$.
\end{Lemma}

\begin{proof}
Let $N(v)=\{v_{1}, v_{2},...,v_{n}\}$. Then $T\setminus \{v\}$ has $n$ connected components. Let $T_{i}$ be the tree of $F\setminus\{v\}$ that contains $v_{i}$. Suppose $i\not=j$. Note that since $T$ is a tree, the unique path from $v_{i}$ to $v_{j}$ must pass through $v$. 
 
To show that $E_{T}=\bigcup_{i=1}^{n} E_{T_{i}}\cup\{\mathfrak{e}\}$, we need only show $E_{T}\subseteq \bigcup_{i=1}^{n} E_{T_{i}}\cup\{\mathfrak{e}\}$ as the other inclusion is clear. Let $\mathfrak{u}\in E_{T}\backslash \{e\}$. Since $E_T$ is an independent set, $\mathfrak{u} \cap \mathfrak{e}=\emptyset$. Suppose there exist two vertices in $\mathfrak{u}$, $u_{i}$ and $u_{j}$, such that $u_{i}$ is a vertex of $T_{i}$ and $u_{j}$ is a vertex of $T_{j}$ for $i\not= j$. Now $v_i$ is connected to $u_i$ in $T_i$, $v_j$ is connected to $u_j$ in $T_j$, and $u_i$ is connected to $u_j$ in $\mathfrak{u}$. Thus there is a path in $T$ from $v_{i}$ to $v_{j}$ that does not pass through $v$, a contradiction. Thus we have that $\mathfrak{u}$ is a subtree of some $T_{i}$ and that $\mathfrak{u}\in E_{T_{i}}$.
\end{proof}

\bigskip

The following lemma extends Lemma~\ref{Vertex Missing 2} to address maximality. It will allow us to use a recursive argument to obtain the K\"{o}nig property for subtree clutters.

\begin{Lemma}\label{Forest Independent Trees}
If $F$ is a forest, $\mathfrak{C}$ a subtree clutter of $F$, $E$ is a maximal set of independent edges of $\mathfrak{C}$, and  $\mathfrak{e}\in E$, then there exists a vertex $v\in \mathfrak{e}$ such that $E\setminus\{\mathfrak{e}\}$ is a maximal set of independent edges for $\mathfrak{C}\setminus \{v\}$.
\end{Lemma}

\begin{proof}
Fix an edge $\mathfrak{e}\in E$ and let $v \in \e$ be any vertex of $\e$. First note that $\e$ is contained in some connected component $T$ of $F$. If $S$ is any other connected component of $F$, then $S$ remains unchanged when passing to $F\setminus \{ v\}$, as does $E_S$. Thus we may assume that $F=T$ is a tree.
Let $T_1, \ldots ,T_n$ be the connected components of $T \setminus \{v\}$. Since $E = \bigcup_{i=1}^n E_{T_i} \cup \{ \e\}$ by Lemma~\ref{Vertex Missing 2} and $E$ is maximal, then $E_{T_i}$ is not maximal for at most one $i$. To see this, suppose two of the sets, say $E_{T_1}$ and $E_{T_2}$ are not maximal. Then there exist independents sets $E_{T_1}'$ and $E_{T_2}'$ in $T_1$ and $T_2$ respectively with $|E_{T_i}'| \geq |E_{T_i}|+1$ for $i=1,2$. Then $E'=E_{T_1}' \cup E_{T_2}' \cup  \bigcup_{i=3}^n E_{T_i}$ is an independent set with cardinality at least $|E|+1$, a contradiction. Notice that a similar argument shows that the unique non-maximal set $E_{T_i}$ must have an element with a nontrivial intersection with $\e$.

Since $\mathfrak{e}$ is a subtree of $T$, there exists a vertex $x_1$ of $\e$ that is a leaf of $\e$. If $E\setminus \{\e \}$ is a maximal set of independent edges for $\C \setminus \{ x_1\}$, the result holds. If not, the unique connected component $T_{i_1}$ of $T\setminus \{x_1\}$ for which $E_{T_{i_1}}$ is not maximal must intersect $\e$. Since $x_1$ is a leaf of $\e$, $T_{i_1}$ is the only connected component of $T\setminus \{x_1\}$ that intersects $\e$, and there is a unique neighbor $x_2$ of $x_1$ in $\e \cap T_{i_1}$. Now consider $T\setminus \{ x_2\}$. Again either the result holds for $x_2$ and we are done, or there is precisely one connected component $T_{i_2}$ of $T\setminus \{ x_2\}$ for which $E_{T_{i_2}}$ is not maximal. Since $T_{i_2}$ must intersect $\e$, there will be a unique vertex $x_3 \in N(x_2) \cap \e \cap T_{i_2}$. We claim that $x_3 \not= x_1$. To see this, form $T \setminus \{x_1x_2\}$. Since $T$ is a tree, deleting one edge produces precisely two connected components, which we will denote by $T_{x_1}$ and $T_{x_2}$ where $T_{x_i}$ contains $x_i$ for $i=1,2$. As before, at most one of $E \cap T_{x_1}$ or $E \cap T_{x_2}$ is not maximal. Notice that $T_{x_1}$ is precisely the connected component of $T\setminus\{ x_2\}$ that contains $x_1$, and $T_{x_2}$ is precisely the connected component of $T\setminus\{ x_1\}$ that contains $x_2$, which is denoted by $T_{i_1}$ above. Since $E_{T_{i_1}}=E_{T_{x_2}}$ is not maximal, then $E_{T_{x_1}}$ is maximal. Thus $x_3 \not= x_1$.  

Now we may repeat the process. Either the result holds for $x_3$, or there is precisely one connected component $T_{i_3}$ of $T\setminus\{x_3\}$ for which $E_{T_{i_3}}$ is not maximal. As before, there is a unique vertex $x_4 \in N(x_3) \cap \e \cap T_{i_3}$. Deleting the edge $x_2x_3$ and following the argument above shows that $x_4 \not= x_2$. Notice that this process produces a path $\{ x_1, x_2, x_3, x_4\}$ in $\e$. This process can be continued to produce a path $\{x_1, \ldots , x_t\}$ in $\e$ for which $E\setminus \e$ is not a maximal set of independent edges for $\C \setminus \{x_i\}$  for $1 \leq i \leq t-1$. Since $\e$ has a finite diameter, the process must terminate, say at $x_t$. Then $E \setminus \{\e\}$ is a maximal set of independent edges for $\C \setminus \{x_t\}$.
\end{proof} 

We are now ready to show that subtree ideals, which form a very general class of monomial ideals associated to forests, satisfy the K\"{o}nig property. This generalizes \cite[Theorem 5.3]{FaridiCM}, as will be seen in  Theorem~\ref{Simplicial Trees are Subtree Clutters}. Notice, however, that subtree ideals, even the special case of path ideals, are not necessarily odd cycle free. See Example~\ref{Tree not Simplicial}.

\begin{Theorem}\label{Trees are Konig}

If $F$ is a forest and $\mathfrak{C}$ is a subtree clutter of $F$, then $\mathfrak{C}$ has the K\"{o}nig property.
\end{Theorem}

\begin{proof}
Let $E=\{\mathfrak{e}_{1},\mathfrak{e}_{2}, \ldots ,\mathfrak{e}_{\beta_{1}}\}$ be a set of independent edges of $\mathfrak{C}$ of maximum cardinality. Recall that $\alpha_{0}\geq \beta_{1}$ where $\alpha_0$ is the minimum cardinality of a minimal vertex cover. 

 By Lemma~\ref{Forest Independent Trees} there exists a vertex $v_{1}$ in $\mathfrak{e}_{1}$ such that $E\setminus\{\mathfrak{e}_{1}\}$ is a set of independent edges of maximum cardinality for the clutter $\mathfrak{C}\setminus\{v_{1}\}$. Again by Lemma~\ref{Forest Independent Trees} there exists a vertex $v_{2}$ in $\mathfrak{e}_{2}$ such that $E\setminus\{\mathfrak{e}_{1},\mathfrak{e}_{2}\}$ is a set of independent edges of maximum cardinality for the clutter $\mathfrak{C}\setminus\{v_{1},v_{2}\}$. This process can be repeated until $E\setminus\{\mathfrak{e}_{1},\mathfrak{e}_{2},\ldots ,\mathfrak{e}_{\beta_{1}}\}=\emptyset$ is a set of independent edges of maximum cardinality for the clutter $\mathfrak{C}\setminus\{v_{1},v_{2},\ldots ,v_{\beta_{1}}\}$. This implies that the clutter $\mathfrak{C}\setminus\{v_{1},v_{2},\ldots ,v_{\beta_{1}}\}$ has no edges, which implies that every edge in $\mathfrak{C}$ must contain a vertex in the set $\{v_{1},v_{2},\ldots ,v_{\beta_{1}}\}$. Therefore $\{v_{1},v_{2},\ldots ,v_{\beta_{1}}\}$ is a vertex cover of $\C$, so $\alpha_{0}\leq \beta_{1}$. Thus $\alpha_{0}=\beta_{1}$.
\end{proof}

\begin{Corollary} \label{PathIdealKonig}
A path ideal of a tree satisfies the K\"{o}nig property.
\end{Corollary}

Now that we have established the K\"{o}nig property for a general class of ideals associated to trees, we will focus on path ideals in order to prove the main result of this section. We first need to generalize the notion of a suspension. Recall that a {\it suspension} of a graph $G'$ on vertices $\{x_1, \ldots , x_n\}$ is a graph $G$ on vertices $\{ x_1, \ldots ,x_n, y_1, \ldots ,y_n \}$ with $E(G)=E(G') \cup \bigcup_{i=1}^n \{ x_iy_i\}$. To form a {\it suspension of length $\ell$} of a graph $G$, instead of adding a single edge $x_iy_i$ to each vertex of $G$, one instead adds a path of length $\ell$ to each vertex. In the case of a suspension of a tree, it will be useful to organize this definition so that the underlying graph $G'$ is implied but not explicitly stated.

\begin{Definition} \label{Suspension Definition}
A tree $T$ with vertex set $V$ is called a suspension of length $\ell$ if $T$ has paths $P=\{p_{1},p_{2},...,p_{\beta_{1}}\}$ all of length $\ell$ such that the paths $p_1, \ldots , p_{\beta_1}$ form a perfect matching and the vertices of each $p_i$ can be ordered $\{x_{i},y_{i_{1},}y_{i_{2}},...,y_{i_{\ell}}\}$ where $\deg(y_{i_{k}})=2$ for all $1\leq k \leq \ell -1$, and $\deg(y_{i_{\ell}})=1$.
\end{Definition}

\begin{Theorem}\label{CM iff Suspension}
Let $T$ be a tree with vertex set $V$ and let $I=I_{\ell}(T)$ be the path ideal of length $\ell$ of $T$. Then $I$ is Cohen-Macaulay if and only if $T$ is a suspension of length $\ell$.
\end{Theorem}

\begin{proof}
We will begin by assuming that $T$ is a suspension of length $\ell$. Using the notation from Definition~\ref{Suspension Definition}, $T$ has a set of paths $P=\{p_{1},p_{2},...,p_{\beta_{1}}\}$ that form a perfect matching.  Moreover, for every $i$, $y_{i_{\ell}}$ has $y_{i_{\ell -1}}$ as its unique neighbor, and for every $1\leq k \leq \ell -1$, the only neighbors of $y_{i_k}$ are $y_{i_{k-1}}$ and $y_{i_{k+1}}$, where $y_{i_0}=x_i$ for ease of notation. Suppose $f$ is any edge of $\C$. Since  $f$ corresponds to a path of length $\ell$ of $T$, if $y_{i_k}\in f$ for some $i, k$, it follows that $y_{i_{k-1}} \in f$ as well. Thus the set $\{ x_1, \ldots , x_{\beta_1}\}$ forms a vertex cover of $\C$. Since any vertex cover must contain at least $\beta_1$ vertices, else it would not cover the paths in $P$, we have that $P$ forms a perfect matching of $\C$ of K\"{o}nig type. 

Now suppose $f_1,f_2$ are any two edges of $\mathfrak{C}$ and fix any $i$. If $f_1 \cap p_i = \emptyset$, then $f_1 \cap p_i \subset f_2 \cap p_i$. The result is similar if $f_2 \cap p_i = \emptyset$. Assume neither intersection is empty. Let $t$ be the greatest integer such that $y_{i_t} \in f_1$ and let $s$ be the greatest integer such that $y_{i_s} \in f_2$. As above, $y_{i_j} \in f_1$ for $0 \leq j \leq t$ and $y_{i_j} \in f_2$ for $0 \leq j \leq s$. It follows that if $t \leq s$, then $f_1 \cap p_i \subset f_2 \cap p_i$ and if $s \leq t$, then $f_2 \cap p_i \subset f_1 \cap p_i$. Thus by \cite[Theorem 2.16]{MRV} we have that $\Delta_{\C}$ is pure shellable, and thus  $\C$ is Cohen-Macaulay (see for example \cite[Theorem III.2.5]{Stanley} or \cite[Theorem 5.3.18]{monalg}).

\medskip

For the converse, suppose $R/I(\C)$ is Cohen-Macaulay. Then $I(\C)$ is unmixed, so every minimal vertex cover of $T$ has the same cardinality $\alpha_0$. Since $T$ is a tree, $\C$ satisfies the K\"{o}nig property by Corollary~\ref{PathIdealKonig}. Thus we can find a set of independent paths $P=\{p_{1} ,p_{2}, \ldots, p_{\alpha_{0}}   \}$ in $T$.  Suppose that $P$ is not a partition of $T$. Then there exists a vertex $v$ of $T$ that is not in any path in $P$. Now $v$ is contained in some minimal vertex cover $U$ of $T$. Since $U$ must also cover each of the paths in $P$, $|U|\geq \alpha_{0} +1$, a contradiction to $I(\C)$ being unmixed.  Hence $P$ is a perfect matching of $T$. 
 
To obtain the desired ordering on each of the paths of $P$, order the vertices of $p_i$ so that $p_i=\{x_{1},x_{2},...,x_{\ell+1}\}$ where $x_j$ is adjacent to $x_{j+1}$ for $1 \leq j \leq \ell$. Now suppose that there is an $s\in \{ 2, \ldots , \ell\}$ with $\deg (x_s) \geq 3$. Then there is an edge $x_sz$ with $z \not\in p_i$. Since $P$ is a partition, $z \in p_q$ for some $q$. Order the vertices of $p_q$ by $\{y_{1},y_{2},...,y_{\ell+1}\}$ where $y_k$ is adjacent to $y_{k+1}$ for $1 \leq k \leq \ell$. Then $z=y_t$ for some $1 \leq t \leq \ell+1$. Notice that since $T$ is a tree, the only edge of the form $x_jy_k$ is $x_sy_t$.  Using a relabling of the paths $p_i$ and $p_q$ if necessary, we may assume $d(y_{1}, y_{t}) \leq d(y_{t}, y_{\ell+1})$, and $d(x_{1}, x_{s}) \leq d(x_{s}, x_{\ell+1})$.  Thus $s, t \leq {\frac{\ell+2}{2}}$. Using symmetry, with the roles of $x_j$ and $y_k$ reversed, we may assume without loss of generality that $t\leq s$. Then $t \leq t+\ell-s+1 \leq \ell+1$ and $2s-1 \leq \ell+1$, so we may consider the vertices $x_{1},x_{2s-1},$ and $y_{t+\ell-s+1}$.
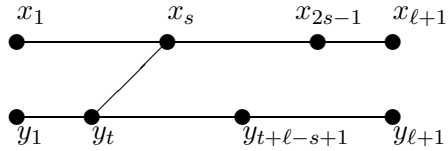
\[
 \begin{matrix}
 
\setlength{\unitlength}{1cm}
\begin{picture}(0,2)

\put(0,0){\line(5,0){5}}

\put(0,1){\line(1,0){5}}

\put(1,0){\line(1,1){1}}

\put(0,-.3){$y_{1}$}
\put(0,0){\circle*{.2}}

\put(1,-.3){$y_{t}$}
\put(1,0){\circle*{.2}}

\put(3,-.3){$y_{t+\ell-s+1}$}
\put(3,0){\circle*{.2}}

\put(5,-.3){$y_{\ell+1}$}
\put(5,0){\circle*{.2}}

\put(0,1.3){$x_{1}$}
\put(0,1){\circle*{.2}}

\put(2,1.3){$x_{s}$}
\put(2,1){\circle*{.2}}

\put(3.7,1.3){$x_{2s-1}$}
\put(4,1){\circle*{.2}}

\put(5,1.3){$x_{\ell+1}$}
\put(5,1){\circle*{.2}}
\end{picture}

&&&&&&&&&

 \end{matrix}
\]

Notice that $d(x_{1}, y_{t+\ell-s+1}) = s-1+1+(t+\ell-s+1-t)=\ell+1$, so there is a path $h_1$ of length $\ell$ connecting $x_1$ and $y_{t+\ell-s}$. Similarly, there is a path $h_2$ of length $l$ connecting $x_{2s-1}$ and $y_{t-\ell+s}$, and a path $h_3$ of length $\ell$ connecting $x_2$ and $y_{t+\ell-s+1}$. Notice that $d(y_{t}, y_{t+\ell-s+1})=\ell-s+1$. Suppose $t \leq \ell-s+1$. Then $d(y_1, y_t) \leq \ell-s$, and $d(x_1, y_1) \leq s-1+1+\ell-s=\ell$ and $d(y_1, x_{2s-1})\leq \ell-s+1+(2s-1-s)=\ell$. Thus every path of length $\ell$ in the induced subtree on $\{ p_i, p_q\}$ must contain one of the vertices $x_1, x_{2s-1}, y_{t+\ell-s+1}$. By examining the paths $h_1, h_2, h_3$, it follows that the set $M=\{ x_1, x_{2s-1}, y_{t+\ell-s+1} \}$ is a minimal vertex cover of the clutter of paths of length $\ell$ of the induced subtree on $\{p_i, p_q \}$. 

Now suppose that $t>\ell-s+1$. Then $1 \leq t-\ell+s-1 <t$, and there is a path of length $\ell$ connecting $x_2$ and $y_{t-\ell+s-1}$. An argument similar to the one above shows that $M=\{ x_1, x_{2s-1}, y_{t+\ell-s+1}, y_{t-\ell+s-1} \}$ is a minimal vertex cover of the clutter of paths of length $\ell$ of the induced subtree on $\{p_i, p_q \}$. In either case, since $M$ is a minimal vertex cover of the clutter of paths of length $\ell$ of the induced subtree on $\{p_i, p_q \}$, $M$ can be extended to a minimal vertex cover $U$ of $\C$. Since $P$ is a partition of $\C$, at least $\alpha_0 -2$ additional vertices are needed to cover the paths in $P$. Thus $|U| \geq |M| + \alpha_0 - 2 > \alpha_0$, which contradicts $I(\C)$ being unmixed. Thus for each $i$, every vertex of $p_i$ that is not an endpoint has degree precisely $2$. 

\medskip

Suppose  that there is a path $p_{i}\in P$ with its endpoints both having degree $2$ or greater. Let $p_{i}$ have the ordered vertices $\{x_{1},x_{2},...,x_{\ell+1}\}$. Then $x_{1}$ and $x_{\ell+1}$ are adjacent to endpoints of distinct paths $p_j$ and $p_q$ in $P$, respectively. 
Using the usual conventions, order the vertices of $p_j$ by $\{y_{1}, y_{2},...,y_{\ell+1} \}$ where  $x_{1}$ is adjacent to $y_{\ell+1}$ and order the vertices of $p_q$ by $\{z_{1},z_{2},...,z_{\ell+1}\}$ where $x_{\ell+1}$ is adjacent to $z_{1}$.

It is easy to see that $M=\{y_{1},x_{1},x_{\ell+1},z_{\ell+1}\}$ is a minimal vertex cover of the clutter of paths of length $\ell$ on the induced subtree of $\{ p_i, p_j, p_q\}$.  As before, $M$ can be extended to a minimal vertex cover $U$ of $\C$. Notice that $|U| \geq |M| + \alpha_0 -3 = \alpha_0 +1$, which is a contradiction to $I(\C)$ being unmixed. Thus at least one endpiont of each element of $P$ must have degree $1$. Hence $P$ satisfies the conditions of Definition~\ref{Suspension Definition} and thus if $I(\C)$ is Cohen-Macaulay then $T$ is a suspension of length $\ell$. 
\end{proof}

Since path ideals can be represented as clutters, they can also be represented as facet ideals of simplicial complexes. In \cite[Corollary 8.3]{FaridiCM}, it is shown that if $\Delta$ is a simplicial tree and $I(\Delta)$ is the ideal whose generators are the facets of $\Delta$, then $I(\Delta)$ is Cohen-Macaulay if and only if $I(\Delta)$ is unmixed, and \cite[Section 6]{FaridiCM} describes the structure of unmixed simplicial trees.
Theorem~\ref{CM iff Suspension} provides a similar characterization of the Cohen-Macaulay property involving clutters whose edges are paths of uniform length of a tree. In \cite[Theorem 2.7]{HeVanTuyl} it is shown that the path ideal of a directed tree is the facet ideal of a simplicial tree, and so Faridi's result applies to this case. However, for a non-directed tree, this need not be the case. In fact, the converse is true, as will be seen in Theorem~\ref{Simplicial Trees are Subtree Clutters}. The following example gives a tree $T$ such that $I_2(T)$ is not a simplicial tree, thus the results of Theorem~\ref{CM iff Suspension} are not implied by \cite[Corollary 8.3]{FaridiCM}.

\begin{Example}\label{Tree not Simplicial} 
Consider the tree $T$ with vertex set $\{a,b,c,d\}$.

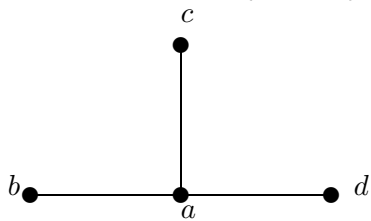
\[
 \begin{matrix}
 \setlength{\unitlength}{1cm}
\begin{picture}(0,2)

\put(-2,0){\line(1,0){4}}
\put(0,0){\line(0,1){2}}
\put(0,-.3){$a$}
\put(0,0){\circle*{.2}}
\put(-2.3,0){$b$}
\put(-2,0){\circle*{.2}}
\put(0,2.3){$c$}
\put(0,2){\circle*{.2}}
\put(2.3,0){$d$}
\put(2,0){\circle*{.2}}
\end{picture}
\\
 \end{matrix}
\]
\end{Example}

Consider the clutter $\C=\{abc, abd, acd\}$ of paths of length $2$ of $T$. In \cite{FaridiCM} Faridi examines the properties of the complex $\Delta$ whose facets are $\{a,b,c\}, \{a,b,d\}, \mbox{ and }\{a,c,d\}$ in order to determine the Cohen-Macaulayness of $I(\Delta)=I(\C)=(abc,abd,acd)$. Note that $\{a,b,c\}\cap\{a,b,d\}=\{ a,b\}$, $\{a,b,c\}\cap\{a,c,d\}=\{a,c\}$, and $\{a,b,d\}\cap\{a,c,d\}=\{a,d\}$, none of which are subsets of each other, so there can be no leaf in $\Delta$, hence $\Delta$ is not a simplicial tree. Notice also that  $c, \{a,b,c\}, b, \{a,b,d\}, d, \{a,c,d\}$ forms an odd cycle in the sense of \cite[Definition 2.11]{CF} and a special odd cycle as defined in \cite{HHTZ}.

Example~\ref{Tree not Simplicial} shows that not every subtree clutter of a tree, or indeed every path ideal of a tree, correpsonds to the facet ideal of a simplicial tree. We now show that the converse does hold. Thus the class of ideals corresponding to subtree clutters subsumes the class of simplicial trees. 

\begin{Theorem}\label{Simplicial Trees are Subtree Clutters}
Suppose $\Delta$ is a simplicial tree and $I=I(\Delta)$. Then $I$ is a subtree ideal of some tree $T$.
\end{Theorem}

\begin{proof}
If there exists a spanning tree $T$ of  the graph formed by the one-skeleton, $\Delta^1$, such that $T \cap H$ is connected for every facet $H$ of $\Delta$, then $I$ is a subtree ideal of $T$, where the subtrees are $H\cap T$ for each facet $H$ of $\Delta$. Such a spanning tree will be referred to as a good spanning tree. If $\Delta$ has a single facet, let $T$ be any spanning tree of $\Delta$. Then $I=I(\Delta)=(M)$ where $M$ is the monomial product of the vertices of $T$ and $T$ is a good spanning tree. Consider an arbitrary simplicial tree $\Delta$. By definition, there exists a leaf $F$ of $\Delta$ and a facet $G$ of $\Delta$ such that $F \cap H \subseteq F \cap G$ for every facet $H$ of $\Delta$. Since $\Delta$ is connected, $F \cap G \neq \emptyset$.  By induction on the number of facets, we may assume $\Delta \setminus \{F\}$ has a good spanning tree $T$. If $F \cap G \cap T$ is connected, select any vertex $x$ in $F \cap G$. For every vertex $y \in F \setminus \{G\}$, add the edge $\{x,y\}$ to $T$ to form a new tree $\widehat{T}$. Notice that $F \cap \widehat{T}$ is connected and $H \cap \widehat{T}=H \cap T$ for every facet $H\neq F$. Thus $\widehat{T}$ is a good spannng tree for $\Delta$. 

If $F\cap G \cap T$ is not connected, we will build a new good spanning tree $T'$ of $\Delta \setminus \{F\}$ such that $F\cap G \cap T'$ is connected, and then proof will follow as above. Assume there exist vertices $a,b$ in $F\cap G$ that are not connected in $F \cap G \cap T$.  Notice that $a$ and $b$ are connected in $G \cap T$, so there exists a path $e_1,e_2, \ldots,e_s$ in $G\cap T$ where $e_i=x_ix_{i+1}$ for some vertices $x_i$ with $a=x_1, \, b=x_{s+1}$. If $d(a,b)=1$, then $a,b$ are connected in $F \cap G \cap T$. Thus we may assume $s \geq 2$ and $d(a,b)$ is minimal among pairs of vertices not connected in $F\cap G \cap T$. By this minimality, $x_i \not \in F$ for $i \neq 1,s+1$.

Assume for every $i$ there exists a facet $H_i \neq G$ with $e_i \in H_i$ such that $H_i$ does not contain both $a$ and $b$. Note that $a \in H_1$ since $x_1=a$. Define $j_1=\max\{i \, | \, a \in H_i\}$ and $k_1=\max\{i \, | \, x_i \in H_{j_1}\}$. Since $b \in H_s$, $a \not\in H_s$, and $e_{j_1}=x_{j_1}x_{j_1+1} \in H_{j_1}$, then $1 \leq j_1 < k_1 < s$. Define $j_2=\max\{i\, | \, x_{k_1} \in H_i\}$ and $k_2=\max\{i \, | \, x_i \in H_{j_2}\}$. Then $j_1 < j_2 <k_2 \leq s+1$. Continue in this manner, until $k_t=s+1$ for some $t$.  By definition, $x_{k_p} \in H_{j_p}\cap H_{j_p+1}$ and $x_{k_p}\not\in H_{j_r}$ for $r \neq j_p,j_p+1$. Also, $a \in F \cap H_{j_p}$ precisely when $p=1$, $b \in F \cap H_{j_p}$ precisely when $p=t$ and $x_i \in F$ only when $x_i=a$ or $x_i=b$. Thus $a, H_{j_1} ,x_{k_1}, H_{j_2} , x_{k_2} ,\ldots , b, F$ is a special cycle of length at least three in $\Delta$. Thus by \cite[Theorem 3.2]{HHTZ} $\Delta$ is not a simplicial tree, a contradiction.  

Thus we may
assume there exists an $i$, such that for every facet $H$ of $\Delta$, if $e_i \in H$, then $a,b\in H$. Note that  $e_i \not\in H$ for every facet $H \neq G$ in $\Delta$ is possible. Define $T'=T\cup \{a,b\} \setminus e_i$.  Suppose $T'$ contains a cycle $C$. The cycle must contain the edge $\{a,b\}$ and cannot contain $e_i$. Thus $C\setminus \{a,b\}$ is a path in $T$ connecting $a$ to $b$ not involving $e_i$. But since $T$ is a tree, there is a unique path connecting $a$ to $b$, a contradiction. Thus $T'$ does not contain a cycle. If $u$ is any vertex, there is a unique path in $T$ between $u$ and $b$. If $e_i$ is not on this path, then the path survives in $T'$.  If $e_i$ is on this path, then $e_i$ is not on the unique path connecting $u$ to $a$. Then there is a unique path from $u$ to $b$ passing through $a$ in $T'$. Thus $u$ is connected to $b$ in $T'$ for every $u$. So $T'$ is a tree. If $e_i \not\in H$ for every facet $H \neq G$, then for every facet $H\neq G$ of $\Delta$,  $T\cap H=T' \cap H$ and so the monomial corresponding to $H$ corresponds to the subtree $H \cap T'$. If $u$ is any vertex of a facet $H$ that contains $a$ and $b$,  then $u$ is connected to $a$ and to $b$ by unique paths in $H \cap T$. Thus as before, $u$ is connected to $b$ in $H \cap T'$ and $T'$ is a good spanning tree of $\Delta \setminus \{F\}$. Notice that $e_i \not\in F \cap G \cap T$, so $F \cap G \cap T' = F \cap G \cap T \cup \{a,b\}$. 

If $F \cap G \cap T'$ is not connected, repeat this process with $T'$ playing the role of $T$. Notice that the edge $e_i$ that is removed in this process is by definition not contained in $F \cap G$, but the edge $\{a,b\}$ that is added is in $F \cap G$. Thus the process strictly increases the number of edges of $F \cap G \cap T$. Since there are finitely many vertices in $F \cap G$, this process is finite and must stop, at which point, $F \cap G \cap T$ is connected, as desired.
\end{proof}

An immediate application of Theorem~\ref{Simplicial Trees are Subtree Clutters} is that Theorem~\ref{Trees are Konig} extends a previously known result about properties of edge ideals of simplicial trees. 
\begin{Corollary}\cite[Theorem 5.3]{FaridiCM}
The edge ideal of a simplicial tree satisfies the K\"{o}nig property.
\end{Corollary}

It is interesting to note that the definition of a subtree ideal can be extended to graphs. If $G$ is any graph, a subtree ideal $I$ of $G$ is a square-free monomial ideal whose generators correspond to subtrees of $G$. Using this definition, it is easy to see that every square-free monomial ideal is a subtree ideal of a graph. For example, if $\Delta$ is the simplicial complex for which $I=I(\Delta)$, then $I$ is a subtree ideal of $\Delta^1$. Note that the graphs need not be unique, so this is not a one-to-one correspondence. However, it does provide a new combinatorial perspective from which to view square-free monomial ideals.

\section{Depths of Path Ideals of Spines}\label{Depths and Stanley}

Although identifying ideals that are Cohen-Macaulay is an important goal, it is also useful to know the depths of ideals that are not Cohen-Macaulay. In general, it can be quite difficult to determine the precise depths. In this section, we give an exact formula for the depth of a path ideal of a tree consisting of a single path. 
By noting that in this special case the directed path ideal is the same as the path ideal, and by using the Auslander-Buchsbaum formula, this depth recovers the projective dimension result of \cite[Theorem 4.1]{HeVanTuyl} which was also recovered in \cite[Corollary 5.1]{BouchatHaO'Keefe}. However, the method of proof allows us to extend the depth result to a bound on the Stanley depths of the ideals, as was done in \cite{StanleyDepthTree} for powers of edge ideals. This bound shows that these ideals are Stanley. 

The primary tool we will employ for computing depths is to form a family of short exact sequences  and then apply the Depth Lemma (see, for example,  \cite[Proposition 1.2.9]{BHer}, or \cite[Lemma 1.3.9]{monalg}). In particular, we will use that if 
$$0 \rightarrow A \rightarrow B \rightarrow C \rightarrow 0$$
is a short exact sequence of finitely generated $R$ modules with homogeneous maps and $\depth (C) > \depth (A)$, then $\depth (B)= \depth (A)$. 
Note that the method used in this section is a variation of the method used in \cite{HaM, morey}, where instead of using the left term of one sequence to form the subsequent sequence, the right hand term is used. Starting with the standard short exact sequence
$$0 \rightarrow R/(I : z) \overset{f}\rightarrow R/I \overset{g}\rightarrow R/(I, z) \rightarrow 0$$
and making judicious choices for $z\in R$, we form a family of sequences 
\begin{equation} \label{sequences}
\begin{matrix}
0 & \rightarrow & R/K_{1} & \rightarrow & R/I & \rightarrow & R/C_{1} & \rightarrow & 0 \\
& & \vdots & & \vdots & & \vdots & & \\
0 & \rightarrow & R/K_{i} & \rightarrow & R/C_{i-1} & \rightarrow & R/C_{i} & \rightarrow & 0 \\
& & \vdots & & \vdots & & \vdots & & \\
0 & \rightarrow & R/K_{s} & \rightarrow & R/C_{s-1} & \rightarrow & R/C_{s} & \rightarrow & 0 \\
\end{matrix}
\end{equation}
where 
$C_0=I$, $K_i=(C_{i-1}: z_i)$, and $C_i=(C_{i-1}, z_i)$ for $1\leq i\leq s$. 
The goal is to find bounds on the depths of $K_i$  for $1 \leq i \leq s$ and for $C_s$. Then applying the Depth Lemma starting with the last sequence and working back to the first will yield a bound on the depth of $R/I$. In this section, it will be easier to describe the sequence $\{ z_i \}$ using a double index, so the ideals playing the roles of $K_{i}$ and $C_{i}$ will be doubly indexed as well. 

A tree that does not branch is traditionally referred to as a {\it path}, however, to avoid the confusion of dealing with path ideals of paths, we will refer to such a graph as a spine. To be precise, we define a {\em spine} of length $n-1$ to be a set of $n$ distinct vertices $x_1, \ldots , x_n$ together with $n-1$ edges $x_ix_{i+1}$ for $1 \leq i \leq n-1$. We denote such a spine by $S_n$
and we will use $R=k[x_1, \ldots ,x_n]$ to denote the polynomial ring associated to $S_n$, or more generally, any graph on $n$ vertices. As subrings of $R$ will be used, define $R_t=k[x_1, \ldots , x_t]$ for $t \leq n$. While working with these ideals, it will often be convenient to work with subideals generated by selected paths. To facilitate this, define $P_{(\ell,1,s)}$ to be the ideal generated by the monomials corresponding to all paths of length $\ell$ of the spine connecting $x_1$ to $x_s$. For example, $P_{(2, 1,5)}=(x_1x_2x_3,x_2x_3x_4,x_3x_4x_5)$. Using this notation, $P_{\ell}(S_n)=P_{(\ell,1,n)}$.

\begin{Lemma} \label{lem:  Depth of Trivial Spines}
Let $S_n$ be a spine on $n$ vertices.  If $n \le \ell$, then $\depth (R/P_{(\ell,1,n)}) = n$.
\end{Lemma}

\begin{proof}
As $\ell \ge n$ we see that $S_n$ does not contain a path of length $\ell$.  Thus $P_{(\ell,1,n)}=P_{\ell}(S) = (0)$ and we have $\depth (R/P_{(\ell,1,n)}) = \depth (R/(0)) = \depth (R) = n$.
\end{proof}

We now fix $\ell$ and $n$. In order to define the monomials that will serve the role of $z_i$ above, it is useful apply the division algorithm to produce unique integers $b$ and $c$ with $0 \leq c < \ell+2$ and  $n-\ell-1=b(\ell+2) +c$. It will often be convenient to write $n= (\ell+1) +b(\ell+2) +c$ throughout the paper. For $1 \leq c \leq \ell+1$, define a sequence $\{ a_{(j,k)}\}$  by $a_{(j,k)}=\prod \limits_{t=n-\ell-k+1}^{n-j-k+1}x_t$ for $1 \leq j \leq \min\{ c, \ell\}$ and $1 \leq k \leq c-j+1$. 
Note that for $c=0$, the sequence is defined to be empty.

\begin{Example} \label{exm:  Spine Sequence Example Pt 1}
Suppose $n=18$ and $\ell=6$.  We then have $b=1$ and $c=3$ so our sequence of monomials $\{a_{(j,k)}\}$ is: 
$$ a_{(1,1)}=x_{12}x_{13} x_{14} x_{15} x_{16} x_{17}, \, a_{(1,2)}=x_{11}x_{12} x_{13} x_{14} x_{15} x_{16}, \, a_{(1,3)}=x_{10}x_{11} x_{12} x_{13} x_{14} x_{15}, $$
$$a_{(2,1)}=x_{12}x_{13} x_{14} x_{15} x_{16}, \, a_{(2,2)}=x_{11}x_{12} x_{13} x_{14} x_{15}, \, a_{(3,1)}=x_{12}x_{13} x_{14} x_{15}
$$
\end{Example}

Using this sequence, we now define the ideals that will play the roles of $C_i$ and $K_i$ in the sequences above. Notice that since the sequence used is doubly indexed, the ideals $C_i$ and $K_i$ will require double indices as well, with the same ranges on the indices as above. We first define the ideals $C_{(j,k)}=(I, a_{(1,1)}, a_{(1,2)}, \ldots, a_{(j,k)})$.
Note that for $c=0$, the sequence was defined to be empty, and the only ideal defined is $C_{(0,k)}=P_{(\ell,1,n)}$ for all $k$. In general, the sequence of $a_{(j,k)}$ was selected so that many of the terms of $C_{(j,k)}=(I, a_{(1,1)}, a_{(1,2)}, \ldots, a_{(j,k)})$ will be redundant. 

Next we define the ideals $K_{(j,k)}$, with the same bounds on $j,k$ as before, by
\begin{equation} \label{eqn:  Definition of Family K}
K_{(j,k)} = \left\{ \begin{array}{ll}
    (C_{(j-1,c-(j-1)+1)}: a_{(j,1)}) & \text{ if } k=1 \\
	(C_{(j,k-1)}: a_{(j,k)}) & \text{ if } k>1
	\end{array}
	\right.
\end{equation}
Notice that each $K_{(j,k)}$ is formed by taking the quotient ideal of the next term in the sequence with the preceeding $C$ ideal.  It is straightforward to obtain an explicit formula for $K_{(j,k)}$. The selection of the sequence $a_{(j,k)}$ was designed so that these quotient ideals will each have two elements of degree one, and these elements will make all paths of length less than $\ell$ redundant as generators.

\begin{Proposition} \label{prp:  Explicit Form of K}
The family of ideals $K_{(j,k)}$ has the following explicit formulation:
\begin{equation} \label{eqn:  Explicit Form of K}
K_{(j,k)} = (P_{(\ell,1,n-\ell-k-1)},x_{n-\ell-k},x_{n-j-k+2})
\end{equation}
\end{Proposition}

\begin{proof}
First notice that for $1 \leq k \leq c$, both $x_{n-\ell-k}a_{(1,k)}$ and $x_{n-k+1}a_{(1,k)}$ are generators of $I=P_{(\ell,1,n)}$,  so $(C_{(1,k-1)}, x_{n-\ell-k}, x_{n-k+1}) \subseteq (C_{(1,k-1)}: a_{(j,k)})$, where $C_{(1,0)}=P_{(\ell,1,n)}$. The other inclusion is straightforward, so removing redundant elements from the list of generators yields the desired result for $j=1$. For $j\geq 2$, first notice that $(a_{(j-1,k+1)} : a_{(j,k)}) = (x_{n-\ell-k})$ and $( a_{(j-1,k)}: a_{(j,k)}) = (x_{n-j-k+2})$. Then the result follows similarly.
\end{proof}

Given this explicit form for $K_{(j,k)}$, it is easy to see that the depth of $K_{(j,k)}$ can be found inductively from the depth of the path ideal of a shorter spine. Thus the Lemma below will allow us to simultaneously control the depth of each of the left hand terms of the series of sequences. The proof is a direct application of  \cite[Lemma 2.2]{morey} and thus is omitted. 

\begin{Lemma} \label{lem:  Isomorphism for K}
For all $j$ and $k$, $\depth (R/K_{(j,k)}) = \depth (R_{n-\ell-k-1}/P_{(\ell,1,n-\ell-k-1)}) + \ell + k - 1$.
\end{Lemma}

We now need to control the depth of the final term of the final sequence. The nature of this proof will allow us to simultaneously handle the case $c=0$, which was omitted above. For convenience, we will denote the final $C_{(j,k)}$ by $I_{(1)}$ and the final $a_{(j,k)}$ by $a_{(1)}$ since the final values of $j$ and $k$ depend on the relationship between $c$ and $\ell$. Explicitly, define

\begin{minipage}[c]{2.5in}
$$I_{(1)}=\left\{ \begin{array}{ll}
I & \text{ if } c=0 \\
I_{(c,1)} & \text{ if }1 \leq c \leq \ell\\
I_{(\ell,2)} & \text{ if } c=\ell+1
\end{array}
\right.$$
\end{minipage}
\hspace{.5in}
\begin{minipage}[c]{2.5in}
$$a_{(1)}=\left\{ \begin{array}{ll}
a_{(c,1)} & \text{ if }1 \leq c \leq \ell\\
a_{(\ell,2)} & \text{ if } c=\ell+1
\end{array}
\right.$$
\end{minipage}

The first two cases to consider follow directly from the definition of $C_{(j,k)}$ and an application of  \cite[Lemma 2.2]{morey}. 

\begin{Lemma} \label{lem:  Isomorphism for I with c=l}
If $c=\ell$, then $\depth (R/I_{(1)}) = \depth (R_{n-\ell-1}/P_{(\ell,1,n-\ell-1)}) + \ell$.
\end{Lemma}

\begin{proof}
Notice that when $c=\ell$, $a_{(c,1)}=x_{n-\ell}$. Also note that $n-\ell-k+1 \leq n-\ell$ and since $k \leq c-j+1, \, n-j-k+1 \geq n-\ell$ when $c=\ell$. Thus $a_{(j,k)}=\prod \limits_{t=n-\ell-k+1}^{n-j-k+1}x_t$ is a multiple of $x_{n-\ell}$ for all $j,k$ when $c=\ell$. Thus $I_{(1)}=C_{(c,1)}=(P_{(\ell,1,n-\ell-1)},x_{n-\ell})$ and the result follows from \cite[Lemma 2.2]{morey}.
\end{proof}

\begin{Lemma} \label{lem:  Isomorphism for I with c=l+1}
If $c=\ell+1$, then $\depth (R/I_{(1)}) = \depth (R_{n-\ell-2}/P_{(\ell,1,n-\ell-2)}) + \ell + 1$.
\end{Lemma}

\begin{proof}
Notice that when $c=\ell+1$, $a_{(\ell,1)}=x_{n-\ell}$ and $a_{(\ell,2)}=x_{n-\ell-1}$. As before, $a_{(j,k)}$ is a multiple of $x_{n-\ell}$ or of $x_{n-\ell-1}$ for all $j,k$, and thus the result follows from \cite[Lemma 2.2]{morey}. 
\end{proof}

Finding the depth of $I_{(1)}$ for $0 \leq c \leq \ell-1$ will require another family of short exact sequences. Define a sequence of monomials by $b_{(h)}= \prod_{t=n-\ell+h}^{n-c}x_t$ for $1 \leq h \leq \ell-c$. 

\begin{Example} \label{exm:  Spine Sequence Example Pt 2}
As in Example~\ref{exm:  Spine Sequence Example Pt 1} assume $n = 18$, $\ell = 6$, $b = 1$, and $c = 3$.  Then $\{b_{(h)}\} =\{ x_{13} x_{14} x_{15}, \, x_{14}x_{15}, \, x_{15}\}$.
\end{Example}

We again form a family of short exact sequences using the sequence $\{b_{(h)}\}$. For convenience, define $J_{(0)}=I_{(1)}$. Now define $J_{(h)}$ and $L_{(h)}$ by $J_{(h)}=(J_{(h-1)},b_{(h)})$ and $L_{(h)}=(J_{(h-1)}:b_{(h)})$. Then as in \eqref{sequences} we have the following family of short exact sequences.

\begin{equation} \label{eqn:  2nd Family of SES's}
\begin{matrix}
0 & \rightarrow & R/L_{(1)} & \rightarrow & R/I_{(1)} & \rightarrow & R/J_{(1)} & \rightarrow & 0 \\
0 & \rightarrow & R/L_{(2)} & \rightarrow & R/J_{(1)} & \rightarrow & R/J_{(2)} & \rightarrow & 0 \\
0 & \rightarrow & R/L_{(3)} & \rightarrow & R/J_{(2)} & \rightarrow & R/J_{(3)} & \rightarrow & 0 \\
& & \vdots & & \vdots & & \vdots & & \\
0 & \rightarrow & R/L_{(l-c-1)} & \rightarrow & R/J_{(l-c-2)} & \rightarrow & R/J_{(l-c-1)} & \rightarrow & 0 \\
0 & \rightarrow & R/L_{(l-c)} & \rightarrow & R/J_{(l-c-1)} & \rightarrow & R/J_{(l-c)} & \rightarrow & 0 \\
\end{matrix}
\end{equation}

Note that for each $h$, $b_{(h)}=x_{n-\ell+h} b_{(h+1)}$ and $a_{(1)}=x_{n-\ell}b_{(1)}$ where $a_{(1)}$ is the final term for the original sequence when $c>0$ and $a_{(1)}= \prod \limits_{t=n-\ell}^{n}x_t$ is the last generator of $I$  when $c=0$. Removing redundant elements from the generating set yields $J_{(\ell-c)} = (P_{(\ell,1,n-c-1)},x_{n-c})$ and $L_{(h)} =(P_{(\ell,1,n-\ell+h-2)},x_{n-\ell+h-1})$. Using these explicit forms of $J_{(\ell-c)}$ and $L_{(h)}$, combined with  \cite[Lemma 2.2]{morey}, we are able to express the depths of all of the left hand terms and the final right hand term of the sequences in \eqref{eqn: 2nd Family of SES's} in terms of the depths of path ideals of shorter spines. Note that by the definition of $b_h$, we will assume $c \leq \ell-1$ whenever we are dealing with $J_{(h)}$ or $L_{(h)}$.

\begin{Lemma} \label{depthLJ}
For all $h$, $\depth (R/L_{(h)}) = \depth (R_{n-\ell+h-2}/P_{(\ell,1,n-\ell+h-2)}) + \ell - h + 1$ and $\depth (R/J_{(\ell-c)}) = \depth (R_{n-c-1}/P_{(\ell,1,n-c-1)}) + c$.
\end{Lemma}

We are now able to prove the main result regarding the depth of a path ideal of a spine.

\begin{Theorem} \label{Depth of Spines}
Let $S$ be a spine of $n$ vertices.  Then 
$$\depth (R/P_\ell(S))=\depth (R/P_{(\ell,1,n)}) = \left\{ \begin{array}{ll}
\ell(b+1) & \text{ if } c=0 \\
\ell(b+1)+c-1 & \text{ if } c>0
\end{array}
\right. .$$
\end{Theorem}

\begin{proof}
We assume $\ell$ is fixed and induct on $n$. If  $n \le \ell$ we have $b=-1$ and $c=n+1$.  By Lemma \ref{lem:  Depth of Trivial Spines} we have $\depth (R/P_{(\ell,1,n)}) = n$ and $\ell(b+1)+c-1=\ell(0)+n+1-1=n$ so the result holds. 

Assume $n \geq \ell+1$. When writing $n=(\ell+1)+b(\ell+2)+c$, notice that for $n\geq 0$, $b = -1$ if and only if $n \leq \ell$. Thus for $n \geq \ell+1$, $b \geq 0$. In the proof that follows, we will be working with $n-t$ for various values of $t$. When $b=0$, this will often result in $n-t \leq \ell$. While this situation can easily be handled using separate caes, allowing $b-1=-1$ creates a more streamlined proof.

Suppose $0 \leq c \leq \ell-1$. Then 
by Lemma~\ref{depthLJ},
$$\depth (R/L_{(h)}) = \depth (R_{n-\ell+h-2}/P_{(\ell,1,n-\ell+h-2)}) + \ell - h + 1,$$
$$\depth (R/J_{(\ell-c)}) = \depth (R_{n-c-1}/P_{(\ell,1,n-c-1)}) + c.$$
Since $1 \leq h \leq \ell-c$ then $0 <c+h \leq \ell$. Now $n-\ell+h-2 = \ell+1 +(b-1)(\ell+2)+c+h$.  By induction, 
$$\depth (R_{n-\ell+h-2}/P_{(\ell,1,n-\ell+h-2)}) = \ell((b-1)+1)+(c+h)-1,$$
so Lemma~\ref{depthLJ} yields 
$$\depth (R/L_{(h)}) = \ell(b)+c+h-1 + \ell - h + 1=\ell(b+1)+c.$$ 
Also by induction, 
$$\depth (R_{n-c-1}/P_{(\ell,1,n-c-1)}) =\ell(b-1+1) + (\ell+1)-1 = \ell(b+1)$$ 
since $n-c-1=(\ell+1)+(b-1)(\ell+2)+\ell+1$,
so $\depth (R/J_{(\ell-c)}) = \ell(b+1) + c$. Now repeated use of the Depth Lemma applied to \eqref{eqn:  2nd Family of SES's} yields $\depth R/I_{(1)} = \ell(b+1)+c$.

Suppose $c=\ell$.  Then by Lemma~\ref{lem:  Isomorphism for I with c=l} we have 
$$\depth (R/I_{(1)}) =  \depth (R_{n-\ell-1}/P_{(\ell,1,n-\ell-1)}) +\ell.$$ 
Then $n-\ell-1=\ell+1+b(\ell+2)+\ell -\ell-1=\ell+1+(b-1)(\ell+2)+\ell+1$. Thus applying the inductive hypothesis with $b'=b-1$ and $c'=\ell+1$ yields 
$$\depth (R_{n-\ell-1}/P_{(\ell,1,n-\ell-1)})=\ell(b-1+1)+(\ell+1)-1,$$ 
so $\depth (R/I_{(1)})=\ell b+\ell+\ell=\ell(b+1)+c$. 

If $c=\ell+1$, then by Lemma~\ref{lem:  Isomorphism for I with c=l+1} we have 
$$\depth (R/I_{(1)}) =\depth(R_{n-\ell-2}/P_{(\ell,1,n-\ell-2)})+\ell +1.$$ 
Then $n-\ell-2=\ell+1+(b-1)(\ell+2)+c$, so by induction, 
$$\depth(R_{n-\ell-2}/P_{(\ell,1,n-\ell-2)})=\ell(b-1+1)+c-1,$$ 
and $\depth (R/I_{(1)}) =\ell(b)+c-1+\ell+1=\ell(b+1)+c$. 

We now have $\depth (R/I_{(1)}) =\ell(b+1)+c$ for all possible values of $c$ . Notice that if $c=0$, we have $P_{(\ell,1,n)}=I_{(1)}$ and $\depth (R/P_{(\ell,1,n)})=\ell(b+1)$ for any $b$, and the result holds. Thus we may now assume $c>0$ for the remainder of the proof.

By Lemma~\ref{lem:  Isomorphism for K}, for all $j,k$
$$\depth (R/K_{(j,k)}) = \depth (R_{n-\ell-k-1}/P_{(\ell,1,n-\ell-k-1)}) + \ell + k - 1.$$ 
Now if $n=(\ell+1) + b(\ell+2)+c$, then $n-\ell-k-1=(\ell+1)+(b-1)(\ell+2) +c-k+1$. Notice that $c-k+1 >0$ since $k \leq c-j+1$. Thus we have 
$$\depth (R_{n-\ell-k-1}/P_{(\ell,1,n-\ell-k-1)}) =\ell(b-1+1)+c-k+1-1=\ell(b)+c-k$$ 
by induction. Then $\depth (R/K_{(j,k)}) =\ell(b)+c-k+\ell+k-1=\ell(b+1)+c-1$. 
Now repeated application of the Depth Lemma to the sequences in \eqref{sequences} yields  $\depth (R/P_{(\ell,1,n)}) =\ell(b+1)+c-1$ when $c>0$. 
\end{proof}

There are some interesting reformulations of the depth  found in Theorem~\ref{Depth of Spines}. They are stated here without proof as the proofs are basic computations and  summation arguments. 

\begin{Corollary}
Theorem~\ref{Depth of Spines} can be reformulated as 

$$\depth(R/P_{(\ell,1,n)})= \left\{
     \begin{array}{lr}
       m\ell, & \text{if }  \ell \leq \frac{n-2m+2}{m} \\
       n-2m+2, & \text{if }  \ell > \frac{n-2m+2}{m}
     \end{array}
   \right.$$
where $m=\lceil \frac{n}{\ell+2}\rceil$, or as $\depth (R/P_{(\ell,1,n)}) = \sum_{i=0}^{\ell-1} \left\lceil \frac{n-i}{\ell+2} \right\rceil$.
\end{Corollary}

Notice that when $\ell$ is large relative to $n$, the depth of $R/P_{(\ell,1,n)}$ is large. If $\ell>n$, then the depth is $n$, as was noted in Lemma~\ref{lem: Depth of Trivial Spines}. However it is intersesting to note that as long as $\ell$ is roughly half of $n$ or larger, the depth remains quite large.

\begin{Corollary}\label{n-2}
If $\ell \geq \frac{n-2}{2}$ , then $\depth (R/P_{(\ell,1,n)}) = n-2$ for $\ell\neq n-1$ and for $\ell=n-1$, $\depth (R/P_{(\ell,1,n)}) =n-1$.
\end{Corollary}

\begin{proof}
Since $\ell \geq \frac{n-2}{2}$, then $b=0$, where $n=(\ell+1)+b(\ell+2)+c$ and $c \leq \ell+1$. By Theorem~\ref{Depth of Spines}, if $c=0$, $\depth (R/P_{(\ell,1,n)})=\ell(b+1)=\ell=n-1$ and if $c >0$, then 
$$\depth (R/P_{(\ell,1,n)})=\ell(b+1)+c-1=\ell+c-1=n-2.$$
\end{proof}

Corollary~\ref{n-2} is particularly interesting when compared to Section~\ref{CMPathIdeals}.  Let $m=\lfloor \frac{n}{\ell+1}  \rfloor$. The set of vertices $M=\{x_{\ell+1},x_{2\ell+2},...,x_{m\ell+m} \}$ forms a minimal vertex cover of minimal cardinality of $P_{(\ell,1,n)}$, so $\height(I)=\lfloor    \frac{n}{\ell+1}   \rfloor$, or  $\dim (R/I) = n - \lfloor \frac{n}{\ell+1} \rfloor$. Thus if $n=2\ell+2$, then by Corollary~\ref{n-2}, $R/ P_{(\ell,1,n)}$ is Cohen-Macaulay. It is also easy to see that when $n=2\ell+2$,  $P_{(\ell,1,n)}$ is the suspension of length $\ell$ of a graph that consists of a single edge connecting two vertices ($x_{n/2}, x_{n/2+1})$. For $n=\ell+1$, $\depth (R/P_{(\ell,1,n)}) = \dim (R/P_{(\ell,1,n)}) = n-1$, which again shows that $R/P_{(\ell,1,n)}$ is Cohen-Macaulay. In this situation, $P_{(\ell,1,n)}$ is the suspension of length $\ell$ of a graph that consists of a single isolated vertex ($x_n$).  For $\ell+1 < n < 2\ell+2$, $\depth (R/P_{(\ell,1,n)}) = n-2$ and $\dim (R/P_{(\ell,1,n)}) = n-1$. 

As remarked before, these results together with the Auslander-Buchsbaum formula, recover the projective dimension found in  \cite[Theorem 4.1]{HeVanTuyl} and in \cite[Corollary 5.1]{BouchatHaO'Keefe}. However, this method of proof has the advantage of also yielding information about the Stanley depth. 
There are three key factors that allow us to extend the depth result to a lower bound on the Stanley depth, or s-depth for brevity. The first two are well known basic facts. If $I$ is a monomial ideal of a polynomial ring $R$ and $y$ is an indeterminate, then 
\begin{equation}\label{s-depth variable}
\sdepth (R[y]/IR[y])=\sdepth (R/I)+1,
\end{equation} 
and $\sdepth R=n$ when $R$ is a polynomial ring in $n$ variables. The third result we will need is that s-depth satisfies a partial version of the Depth Lemma. In particular, it was shown in \cite[Lemma 2.2]{Rauf2} that if
$$0 \rightarrow A \rightarrow B \rightarrow C \rightarrow 0$$
is a short exact sequence of finitely generated $R$ modules then 
$$\sdepth (B) \geq \min \{ \sdepth (A), \sdepth (C)\}.$$
Now by carefully examining the proof of Theorem~\ref{Depth of Spines}, we are able to extend the result to a lower bound on the s-depth of the path ideal of a spine. Note that the explicit calculations closely follow those of Theorem~\ref{Depth of Spines} and so details have been condensed in the proof. 

\begin{Theorem}\label{StanleyDepths}
Let $S_n$ be a spine on $n$ vertices.  Then 
$$\sdepth (R/P_{\ell}(S_n))=\sdepth (R/P_{(\ell,1,n)}) \geq \left\{ \begin{array}{ll}
\ell(b+1) & \text{ if } c=0 \\
\ell(b+1)+c-1 & \text{ if } c>0
\end{array}
\right. .$$
\end{Theorem}

\begin{proof}
We assume $\ell$ is fixed and induct on $n$. Write $n=(\ell+1)+b(\ell+2)+c$. If  $n \le \ell$, $\sdepth (R/P_{(\ell,1,n)}) =\sdepth (R)= n$ and $\ell(b+1)+c-1=\ell(0)+n+1-1=n$ and the result holds. 
 Define the sequences $a_{(j,k)}$ and $b_{(h)}$ and the related ideals $K_{(j,k)}$, $C_{(j,k)}$, $L_{(h)}$, $J_{(h)}$ and $I_{(1)}$ as before. By Proposition~\ref{prp:  Explicit Form of K}, 
$$\sdepth (R/K_{(j,k)})=\sdepth(R/P_{(\ell,1,n-\ell-k-1)})+\ell+k-1,$$ 
and by induction 
$$\sdepth (R/P_{(\ell,1,n-\ell-k-1)}) \geq \ell(b-1+1)+c-k+1-1=\ell(b)+c-k,$$
 so 
$\sdepth (R/K_{(j,k)}) \geq \ell(b)+c-k+\ell+k-1=\ell(b+1)+c-1$. 

If $c=\ell$ or $c=\ell+1$, then as in Lemma~\ref{lem:  Isomorphism for I with c=l} or Lemma~\ref{lem:  Isomorphism for I with c=l+1} with  \cite[Lemma 2.2]{morey} replaced by~\eqref{s-depth variable}, 
$$\sdepth (R/I_{(1)}) =  \sdepth (R_{n-\ell-1}/P_{(\ell,1,n-\ell-1)}) + \ell$$ 
when $c=\ell$, and when $c=\ell+1$, 
$$\sdepth (R/I_{(1)}) =\sdepth(R_{n-\ell-2}/P_{(\ell,1,n-\ell-2)})+\ell +1.$$ 
In either case, applying the inductive hypothesis as in Theorem~\ref{Depth of Spines} yields  
$$\sdepth (R/I_{(1)}) \geq \ell(b+1)+c.$$ 

Suppose $0 \leq c \leq \ell-1$. Then as in Lemma~\ref{depthLJ}  with  \cite[Lemma 2.2]{morey} replaced by~\eqref{s-depth variable}, 
$$\sdepth (R/L_{(h)}) = \sdepth (R_{n-\ell+h-2}/P_{(\ell,1,n-\ell+h-2)}) +\ell - h + 1,$$ 
 $$\sdepth (R/J_{(\ell-c)}) = \sdepth (R_{n-c-1}/P_{(\ell,1,n-c-1)}) + c.$$
As in Theorem~\ref{Depth of Spines}, applying the inductive hypothesis yields 
$$\sdepth (R_{n-\ell+h-2}/P_{(\ell,1,n-\ell+h-2)}) \geq \ell((b-1)+1)+(c+h)-1,$$ 
so  $\sdepth (R/L_{(h)}) \geq \ell(b+1)+c$. Also by induction 
$$\sdepth (R_{n-c-1}/P_{(\ell,1,n-c-1)}) \geq \ell(b-1+1) + (\ell+1)-1 =\ell(b+1),$$ 
so $\sdepth (R/J_{(\ell-c)}) \geq \ell(b+1) + c$. Now repeated use of  \cite[Lemma 2.2]{Rauf2} applied to \eqref{eqn:  2nd Family of SES's} yields 
$$\sdepth R/I_{(1)} \geq \ell(b+1)+c.$$
Notice that if $c=0$, we have $P_{(\ell,1,n)}=I_{(1)}$ and $\sdepth (R/P_{(\ell,1,n)})\geq \ell(b+1)$ for any $b$, and the result holds. For $c>0$ repeated application of \cite[Lemma 2.2]{Rauf2} to the sequences in \eqref{sequences} yields  
$\sdepth (R/P_{(\ell,1,n)}) \geq \ell(b+1)+c-1.$ 
\end{proof}

A monomial ideal $I$ is a Stanley ideal if the Stanley conjecture holds for $I$. That is, if $\sdepth (R/I) \geq \depth (R/I)$. Due to the general difficulty of computing the Stanley depth, very few classes of Stanley ideals are know. It is interesting to note that Theorem~\ref{StanleyDepths} provides a new class of Stanley ideals.

\begin{Corollary}\label{StanleyIdeals}
Let $S$ be a spine of $n$ vertices.  Then $P_{\ell}(S)$ is a Stanley ideal.
\end{Corollary}

\begin{proof}
This follows directly from Theorems~\ref{Depth of Spines} and~\ref{StanleyDepths}.
\end{proof}

\section{ACKNOWLEDGMENTS}
We gratefully acknowledge the computer algebra system Macaulay 2 \cite{M2} which was invaluable in our work on this paper. The authors of this paper were supported by NSF grant (DMS 1005206) during the intital phase of research. We thank NSF and Texas State University for their support. We wish to thank the other students and faculty mentors working on this grant for helpful suggestions and ideas during various group discussions.

\bibliographystyle{plain}

\end{document}